\documentclass[a4paper, 12pt]{amsart}

\usepackage{amssymb}
\usepackage{amsmath}
\usepackage{amscd}
\usepackage{graphicx}
\usepackage{amsthm}
\usepackage{color}
\usepackage{enumerate}
\usepackage[all]{xy}
\usepackage{hyperref}

\theoremstyle{plain}
\newtheorem{Theorem}{Theorem}[section]

\newtheorem{Lemma}[Theorem]{Lemma}

\theoremstyle{definition}
\newtheorem{Defi}[Theorem]{Definition}

\newtheorem{fact}[Theorem]{Fact}

\newtheorem*{mainthm}{Main Theorem}

\DeclareMathOperator{\reg}{reg}

\DeclareMathOperator{\ch}{char}

\def\NZQ{\mathbb}
\def\NN{{\NZQ N}}

\def\opn#1#2{\def#1{\operatorname{#2}}}
\opn\im{im}
\opn\mat{m}

\begin{document}

\title{A Characterization of Edge Ideals with $\reg(R/I(G)) = 3$}
\author{Akane Kanno}
\email{2520021750@campus.ouj.ac.jp}

\maketitle

\begin{abstract}
Let $G$ be a graph and $I(G)$ its edge ideal.
In this paper, we give a complete characterization of the graphs $G$
for which $\reg(R/I(G)) = 3$.
\end{abstract}

\thispagestyle{empty}


\section*{Introduction}

Let $R = K[x_1, \ldots, x_n]$ denote the polynomial ring in $n$ variables over a field $K$,
graded by $\deg x_i = 1$.
Let $I \subset R$ be a homogeneous ideal with $\dim(R/I) = d$.
When the minimal graded free resolution of $R/I$ takes the form
\[
\mathbb{F}_{R/I} : 0 \to \bigoplus_{j \ge 0} R(-j)^{\beta_{p,j}}
\to \cdots
\to \bigoplus_{j \ge 0} R(-j)^{\beta_{1,j}}
\to R \to R/I \to 0,
\]
the \textbf{(Castelnuovo--Mumford) regularity} $\reg(R/I)$ of $R/I$ is defined by
\[
\reg(R/I) := \max\{ j - i \mid \beta_{i,j} \neq 0 \}.
\]

Let $G = (V(G), E(G))$ be a finite simple undirected graph with vertex set
$V(G) = \{x_1, \ldots, x_n\}$ and edge set $E(G)$.
Throughout this paper, all graphs are assumed to be finite, simple, and undirected.
Set $R = K[V(G)] = K[x_1, \ldots, x_n]$, and define the \textit{edge ideal} of $G$ by
\[
I(G) = \left( x_i x_j : \{x_i, x_j\} \in E(G) \right) \subset R.
\]

Edge ideals occupy a central position in commutative algebra and combinatorics,
and have been studied extensively.
Of particular interest is the relationship between algebraic invariants of edge ideals
and combinatorial invariants of graphs.
Among such invariants, the Castelnuovo--Mumford regularity of an edge ideal is one of
the most fundamental---alongside dimension and depth---and has attracted a great deal
of attention. (see, e.g., \cite{BBH}, \cite{H} and \cite{SF}). Representative results are as follows.

\begin{Theorem}[\cite{W} Lemma~7, \cite{HVT}, Theorem~6.7, and \cite{K}, Lemma~2.2]
Let $\mat(G)$ denote the matching number of $G$, and $\im(G)$ the induced matching number
of $G$. Then the following inequalities hold:
\[
\im(G) \leq \reg\!\left(R/I(G)\right) \leq \mat(G).
\]
\end{Theorem}

\begin{Theorem}[\cite{F}]\label{reg2}
For a graph $G$, $\reg(R/I(G)) = 1$ if and only if the complement $G^c$ is chordal.
\end{Theorem}

\begin{Theorem}[\cite{F}]\label{reg3}
For a graph $G$, $\reg(R/I(G)) \ge 2$ if and only if $G^c$ contains a cycle of length
at least $4$.
\end{Theorem}

While Theorems~\ref{reg2} and~\ref{reg3} are established,
the case $\reg(R/I(G)) \ge 3$ was only partially understood:
for bipartite graphs, the following result was known.

\begin{Theorem}[\cite{FRG}, Theorem~3.1]
Let $G$ be a bipartite graph. If $\reg(R/I(G)) \ge 3$, then the bipartite complement
$G^{bc}$ contains a cycle of length at least $6$.
\end{Theorem}

In this paper, we extend this result to arbitrary graphs.

\begin{mainthm}[Theorem~\ref{main}]
Let $G$ be a graph.
\begin{enumerate}
  \item Over any field $k$, $\reg(R/I(G)) \ge 3$ if and only if $G^c$ contains
        a triangulation of an orientable closed $2$-manifold as a subgraph.
  \item Over a field $k$ with $\ch(k) = 2$, $\reg(R/I(G)) \ge 3$ if and only if
        $G^c$ contains a triangulation of a closed $2$-manifold (orientable or not)
        as a subgraph.
\end{enumerate}
\end{mainthm}


\section{Preliminaries}

In this section we collect several lemmas needed for the proof of
Theorem~\ref{main}.

\subsection*{Graph theory}

We begin with a brief review of the graph-theoretic terminology used in this paper.

A graph $G$ is called \textit{bipartite} if there exists a partition
$V(G) = U \cup W$ with $U \cap W = \emptyset$ such that every edge
$e \in E(G)$ satisfies $e \cap U \neq \emptyset$ and $e \cap W \neq \emptyset$.
More generally, $G$ is called an \textit{$n$-partite graph} if there is a partition
$V(G) = U_1 \cup U_2 \cup \cdots \cup U_n$ into pairwise disjoint sets such that
for every edge $e \in E(G)$ there exist distinct indices $i \neq j$ with
$e \cap U_i \neq \emptyset$ and $e \cap U_j \neq \emptyset$.
An $n$-partite graph $G$ is called \textit{complete $n$-partite} if
$E(G) = \bigl\{\{x_i, x_j\} : x_i \in U_i,\, x_j \in U_j,\, i \neq j\bigr\}$,
and is denoted $K_{|U_1|, |U_2|, \ldots, |U_n|}$.
In particular, the complete tripartite graph $K_{2,2,2}$ is called the
\textit{octahedral graph}.

The \textit{complement} of a graph $G$ is the graph $G^c$ defined by
$V(G^c) = V(G)$ and
$E(G^c) = \bigl\{\{x_i, x_j\} : x_i, x_j \in V(G),\, i \neq j\bigr\} \setminus E(G)$.

For a bipartite graph $G$ with bipartition $V(G) = U \cup W$, the
\textit{bipartite complement} $G^{bc}$ is the graph with $V(G^{bc}) = V(G)$ and
$E(G^{bc}) = \bigl\{\{x_i, x_j\} : x_i \in U,\, x_j \in W\bigr\} \setminus E(G)$.

For a graph $G$ and a subset $W \subset V(G)$, the \textit{induced subgraph}
$G_{|W}$ is defined by $V(G_{|W}) = W$ and
$E(G_{|W}) = \bigl\{e = \{x,y\} \in E(G) : x, y \in W\bigr\}$.
When $W = V(G) \setminus \{x\}$, we write $G - x := G_{|W}$.

A \textit{walk} from $x_s$ to $x_g$ in $G$ is a sequence
$x_s x_{p_1} x_{p_2} \cdots x_{p_m} x_g$ such that \\
$\{x_s, x_{p_1}\}, \{x_{p_1}, x_{p_2}\}, \ldots, \{x_{p_m}, x_g\} \in E(G)$.
A graph $G$ is \textit{connected} if for any $x, y \in V(G)$ there exists a walk
from $x$ to $y$ in $G$.

A graph $G$ is \textit{$2$-connected} if $G - x$ is connected for every vertex $x$.
More generally, for an integer $n \ge 2$, $G$ is \textit{$n$-connected} if
$G - x$ is $(n-1)$-connected for every vertex $x$.

A walk $x_s x_{p_1} x_{p_2} \cdots x_{p_m} x_g$ with $x_s = x_g$ and all other
vertices distinct is called a \textit{cycle} of length $m+1$.
A cycle $C$ in $G$ is an \textit{induced cycle} if $G_{|V(C)} = C$.
A subset $U \subset V(G)$ is called an \textit{anticycle} of $G$ if
$(G_{|U})^c$ is an induced cycle of length at least $4$.

For vertices $x$ and $y$, we say $x$ and $y$ are \textit{adjacent} if
$\{x,y\} \in E(G)$.
The \textit{open neighborhood} of $x$ is
$N_G(x) = \{y \in V(G) : \{x,y\} \in E(G)\}$,
and the \textit{closed neighborhood} is $N_G[x] = N_G(x) \cup \{x\}$.
The \textit{degree} of $x$ is $d_G(x) = |N_G(x)|$. 
And we denote $G_x := G_{|V(G) \setminus N_G[x]}$. 

A set $M = \{e_1, \ldots, e_s\} \subset E(G)$ is a \textit{matching} in $G$ if
$e_i \cap e_j = \emptyset$ for all distinct $e_i, e_j \in M$.
A matching $M$ is an \textit{induced matching} if $G_{|V(M)} = M$ (i.e., $M$ forms
the entire edge set of the induced subgraph on $V(M) = \bigcup_{e \in M} e$).
The maximum size of a matching is the \textit{matching number} $\mat(G)$, and the
maximum size of an induced matching is the \textit{induced matching number} $\im(G)$.

For an edge $e = \{a,b\} \in E(G)$, the \textit{contraction} $G/e$ is the graph with
\[
V(G/e) = (V(G) \setminus \{a,b\}) \cup \{v\},
\]
\[
E(G/e) = E\!\left(G_{|V(G)\setminus\{a,b\}}\right)
         \cup \bigl\{\{v,w\} : w \in (N_G(a) \cup N_G(b)) \setminus \{a,b\}\bigr\}.
\]

\bigskip

\subsection*{Regularity of edge ideals and monomial ideals}

We next recall several useful results on the regularity of edge ideals and,
more generally, of monomial ideals.

The following bounds on regularity are well known.

\begin{Theorem}[\cite{HVT}, Theorem~6.7, and \cite{K}, Lemma~2.2]
\label{matchbound}
Let $\mat(G)$ be the matching number and $\im(G)$ the induced matching number of $G$.
Then
\[
\im(G) \leq \reg\!\left(R/I(G)\right) \leq \mat(G).
\]
\end{Theorem}

In particular, if $G^c = K_{2,2,2}$ then $\mat(G) = \im(G) = 3$, so
$\reg(R/I(G)) = 3$.

The following lemma on the regularity of short exact sequences is standard.

\begin{Lemma}[\cite{Peeva}, Proposition~18.6]\label{regslemma}
Let
\[
0 \rightarrow U \rightarrow U' \rightarrow U'' \rightarrow 0
\]
be a short exact sequence of graded finitely generated $R$-modules with
degree-zero homomorphisms. Then
\begin{enumerate}
    \item If $\reg(U') > \reg(U'')$, then $\reg(U) = \reg(U')$.
    \item If $\reg(U') < \reg(U'')$, then $\reg(U) = \reg(U'') + 1$.
    \item If $\reg(U') = \reg(U'')$, then $\reg(U) \le \reg(U'') + 1$.
\end{enumerate}
\end{Lemma}

Consider the following short exact sequence, where $I$ is an ideal in $R$ and $x$ is a
variable:
\[
0 \rightarrow \frac{R}{(I : x)}(-1) \rightarrow \frac{R}{I} \rightarrow \frac{R}{(I, x)} \rightarrow 0.
\]
Applying Lemma~\ref{regslemma} to this sequence gives the following bound.

\begin{Theorem}\label{regbound}
For an ideal $I \subset R = k[x_1, \dots, x_n]$ and a variable $x$,
\[
\reg\!\left(\frac{R}{I}\right) \le \max\!\left\{\reg\!\left(\frac{R}{(I,x)}\right),\, \reg\!\left(\frac{R}{(I:x)}\right) + 1\right\}.
\]
\end{Theorem}

For monomial ideals, equality holds:

\begin{Theorem}[\cite{DHS}, Lemma~2.10]\label{LemC}
For a monomial ideal $I \subset R = k[x_1, \dots, x_n]$ and a variable $x$,
\[
\reg\!\left(\frac{R}{I}\right) = \reg\!\left(\frac{R}{(I,x)}\right)
\quad\text{or}\quad
\reg\!\left(\frac{R}{I}\right) = \reg\!\left(\frac{R}{(I:x)}\right) + 1.
\]
\end{Theorem}

In the case of edge ideals, this specializes as follows.

\begin{Lemma}\label{linkreg}
Let $G$ be a graph with edge ideal $I(G)$. For any vertex $x \in V(G)$, exactly one
of the following equalities holds:
\begin{enumerate}
    \item $\reg\!\left(\dfrac{R}{I(G)}\right) = \reg\!\left(\dfrac{R}{I(G-x)}\right).$
    \item $\reg\!\left(\dfrac{R}{I(G)}\right) = \reg\!\left(\dfrac{R}{I(G_{|V(G)\setminus N_G[x]})}\right) + 1.$
\end{enumerate}
\end{Lemma}

\begin{proof}
Since $I(G)$ is generated by quadratic monomials, we have
$(I(G) : x) = I(G_{|V(G) \setminus N_G[x]})$ and $(I(G), x) = I(G - x)$,
from which the result follows by Theorem~\ref{LemC}.
\end{proof}

The following lemma will also be needed.

\begin{Lemma}[\cite{HT}, Lemma~3.2] \label{regadd}
	Let $R_1 = K[x_1, \dots , x_m]$ and $R_2 = K[y_1, \dots, y_n]$ be polynomial rings over a field $K$. Let $I_1$ be a nonzero homogeneous ideal of $R_1$ and $I_2$ that of $R_2$.
	Let $R = R_1 \otimes_K R_2 = K[x_1, \dots, x_m, y_1, \dots, y_n]$ and regard $I_1 + I_2$ as a homogeneous ideal of $R$. Then
	$\reg(R/(I_1 + I_2)) = \reg(R_1/I_1) + \reg(R_2/I_2)$. 
\end{Lemma}

We now recall the notion of a Betti splitting.
Let $\Gamma(I)$ denote the minimal system of monomial generators of $I$.

\begin{Defi}[\cite{CHA}, Definition~1.1]
Let $I$, $I'$, and $I''$ be monomial ideals such that $\Gamma(I)$ is the disjoint union
of $\Gamma(I')$ and $\Gamma(I'')$.
Then $I = I' + I''$ is a \textit{Betti splitting} if
\[
\beta_{i,j}(I) = \beta_{i,j}(I') + \beta_{i,j}(I'') + \beta_{i-1,j}(I' \cap I'')
\quad\text{for all } i \in \NN \text{ and all degrees } j,
\]
where $\beta_{i,j}(I)$ denotes the $(i,j)$-th graded Betti number of $I$.
\end{Defi}

\begin{Theorem}[\cite{CHA}, Corollary~2.2]\label{reg-pd}
Let $I = I' + I''$ be a Betti splitting. Then
\begin{enumerate}
    \item[(a)] $\reg(I) = \max\{\reg(I'),\, \reg(I''),\, \reg(I' \cap I'') - 1\}$, and
    \item[(b)] $\operatorname{pd}(I) = \max\{\operatorname{pd}(I'),\, \operatorname{pd}(I''),\, \operatorname{pd}(I' \cap I'') + 1\}$.
\end{enumerate}
\end{Theorem}

\begin{Defi}[\cite{CHA}, Definition~2.6]\label{d.onevar}
Let $I$ be a monomial ideal in $R = k[x_1, \dots, x_n]$.
Let $I'$ be the ideal generated by all elements of $\Gamma(I)$ divisible by $x_i$,
and let $I''$ be the ideal generated by all remaining elements of $\Gamma(I)$.
We call $I = I' + I''$ an \textit{$x_i$-partition} of $I$.
If this partition is also a Betti splitting, we call it an \textit{$x_i$-splitting}.
\end{Defi}

\begin{Lemma}[\cite{CHA}, Corollary~2.7]\label{bettispli}
Let $I = I' + I''$ be an $x_i$-partition of $I$ in which every element of $\Gamma(I')$
is divisible by $x_i$. If $\beta_{i,j}(I' \cap I'') > 0$ implies $\beta_{i,j}(I') = 0$
for all $i$ and all multidegrees $j$, then $I = I' + I''$ is a Betti splitting.
In particular, if the minimal graded free resolution of $I'$ is linear, then
$I = I' + I''$ is a Betti splitting.
\end{Lemma}

The following formula, due to Hochster, computes the Betti numbers of a
Stanley--Reisner ring in terms of the simplicial homology of the complex.

\begin{Lemma}[Hochster's formula, cf.~\cite{MS}]\label{hoc}
Let $k[\Delta] = R/I(\Delta)$ be the Stanley--Reisner ring of a simplicial complex
$\Delta$. Then
\[
\beta_{i,j}(k[\Delta]) = \sum_{\substack{W \subset V(\Delta) \\ |W| = j}}
\dim_k \widetilde{H}_{j-i-1}(\Delta_W;\, k),
\]
where $\Delta_W$ denotes the restriction of $\Delta$ to $W$.
\end{Lemma}

\bigskip

\subsection*{Closed $2$-manifolds}

Finally, we recall two standard facts about closed $2$-manifolds (surfaces) and their
triangulations.

\begin{fact}[cf.~\cite{G}]\label{2mani}
A graph $\Gamma$ is the $1$-skeleton of a triangulation of a closed $2$-manifold if and
only if the link of every vertex of $\Gamma$ (i.e., the neighborhood subgraph induced by
$N_\Gamma(v)$) is a cycle.
\end{fact}

\begin{fact}[cf.~\cite{Ha}]\label{2homo}
Let $\Sigma$ be a triangulation of a closed surface, and let $k$ be a field.
Here $S^2$ denotes the $2$-sphere, $\Sigma_g$ the orientable surface of genus $g$,
and $N_g$ the non-orientable surface of genus $g$. Then
\begin{equation}
\widetilde{H}_2(\Sigma;\, k) \cong
\begin{cases}
k & \text{if } \Sigma = S^2 \text{ or } \Sigma_g \text{ with } \ch(k) \neq 2, \\
0 & \text{if } \Sigma = N_g \text{ with } \ch(k) \neq 2,
\end{cases}
\end{equation}
\begin{equation}
\widetilde{H}_2(\Sigma;\, k) \cong k \quad \text{if } \ch(k) = 2
\text{ (for any closed surface } \Sigma\text{)}.
\end{equation}
\end{fact}


\section{Main Result}

\begin{Theorem}\label{main}
Let $G$ be a graph and $k$ a field.
\begin{enumerate}
  \item Over any field $k$, $\reg(R/I(G)) \ge 3$ if and only if $G^c$ contains a
        triangulation of an orientable closed $2$-manifold as a subgraph.
  \item If $\ch(k) = 2$, then $\reg(R/I(G)) \ge 3$ if and only if $G^c$ contains a
        triangulation of a closed $2$-manifold (orientable or not) as a subgraph.
\end{enumerate}
\end{Theorem}

The proof of Theorem~\ref{main} proceeds by showing that if $G$ satisfies
$\reg(R/I(G)) \ge 3$ and is vertex-minimal with respect to this property, then the
combinatorial structure of $G$ is severely constrained.

\begin{proof}[Proof of Theorem~\ref{main}]

Let $G$ be vertex-minimal among graphs satisfying $\reg(R/I(G)) \ge 3$.
That is, $\reg(R/I(G_{|W})) \le 2$ for every proper subset $W \subsetneq V(G)$.

For a vertex $x \in V(G)$, 
we claim that $\reg(R/I(G_x)) = 2$ for every vertex $x$, and in particular
$G_x$ contains an anticycle as an induced subgraph.

Indeed, if $G_x$ contains no anticycle, then $\reg(R/I(G_x)) = 1$ by
Theorem~\ref{reg3}, and Lemma~\ref{linkreg} gives
$\reg(R/I(G)) = \reg(R/I(G - x))$, contradicting vertex-minimality.

We now prove by induction on $|V(G)|$ that $G_x$ is itself an anticycle
for every vertex $x$.

Since $\reg(R/I(G)) \ge 3$ implies $\im(G) \ge 3$ by Theorem~\ref{matchbound},
$G$ contains at least $6$ vertices.
If $|V(G)| = 6$, then $G$ must be the complement of the octahedral graph. 
If $G$ is a complement of the octahedral, then $G_x$ consists of two disjoint edges, which from the complement of $4$-cycle for any vertex $x \in V(G)$. 
So we may assume $|V(G)| \ge 7$.

Fix a pair $a, b \in V(G)$ with $\{a, b\} \notin E(G)$.
Define a graph $G'$ by
\[
V(G') = V(G) \cup \{v_1, v_2\},
\]
\[
E(G') = E(G)
       \cup \bigl\{\{v_i, w\} : i \in \{1,2\},\, w \in N_G(a) \cap N_G(b)\bigr\}
       \cup \bigl\{\{v_1, v_2\}\bigr\}.
\]

\medskip

\textbf{Step 1.} We show that $a, b$ can always be chosen so that
$N_G(a) \cap N_G(b) \neq \emptyset$.

Suppose for contradiction that $N_G(a) \cap N_G(b) = \emptyset$ for every
non-edge $\{a,b\} \notin E(G)$.
Then for any edge $\{a,b\} \in E(G^c)$, we have
$N_{G^c}(a) \cup N_{G^c}(b) = V(G) \setminus \{a,b\}$.

By the argument above, the induced subgraph $G^c_{|N_{G^c}(a)}$,
which coincides with $(G_a)^c$, contains an induced cycle $C = b_1 b_2 b_3 
\cdots b_{l} b_1$.

Since $b_1 \in N_{G^c}(a)$, we have $\{a, b_1\} \notin E(G)$,
so by assumption $N_G(a) \cap N_G(b_1) = \emptyset$.
Combined with $N_{G^c}(a) \cup N_{G^c}(b_1) = V(G)\setminus\{a,b_1\}$,
every vertex of $N_G(a)$ lies in $N_{G^c}(b_1)$; pick any $a' \in N_G(a)$.

Applying the same argument to each adjacent pair $\{b_i, b_{i+1}\} \in E(C) \subset E(G^c)$,
we find that $a' \in N_{G^c}(b_i)$ for all $b_i \in C$.
Since $C$ is an induced cycle in $G^c$ and $a'$ is adjacent in $G^c$ to all vertices of $C$,
the cycle $C$ has length exactly $4$.

It follows that $a, a', b_1, b_2, b_3, b_4$ span an induced octahedral graph in $G^c$,
so vertex-minimality forces $|V(G)| = 6$, a contradiction.

\medskip
\textbf{Step 2.} We show $\reg(R'/I(G')) = 3$, where $R' = R[v_1, v_2]$.

Consider the two short exact sequences
\[
0 \to \frac{R'}{(I(G') : v_2)}(-1) \to \frac{R'}{I(G')} \to \frac{R'}{(I(G'), v_2)} \to 0,
\]
\[
0 \to \frac{R'}{((I(G'), v_2) : v_1)}(-1) \to \frac{R'}{(I(G'), v_2)} \to \frac{R'}{(I(G'), v_1, v_2)} \to 0.
\]
Since $\{v_1, v_2\} \in E(G')$, we have
$G'_{v_2} = G_a \cup G_b \subsetneq G$, so
$\reg(R'/(I(G') : v_2)) + 1 = 3$.
Similarly, $\reg(R'/((I(G'), v_2) : v_1)) + 1 = 3$.
Moreover, $(I(G'), v_1, v_2) = I(G)$ because $(I(G'), v_1, v_2) = I(G' - v_1 -v_2) = I(G)$, so by Lemma~\ref{linkreg},
$\reg(R'/(I(G'), v_2)) = 3$, and applying the first sequence gives
$\reg(R'/I(G')) = 3$.

\medskip
\textbf{Step 3.} Set $G'' = G' - b$. We show $\reg(R'/I(G'')) = 3$.

Consider the sequence
\[
0 \to \frac{R'}{(I(G') : b)}(-1) \to \frac{R'}{I(G')} \to \frac{R'}{(I(G'), b)} \to 0.
\]
Since $\{b, v_i\} \notin E(G')$ and $N_{G'}(v_i) = \{v_j\} \cup (N_G(a) \cap N_G(b))$,
we get $G'_b = G_b \cup \{v_1, v_2\}$.
By Lemma\ref{regadd}, 
$\reg(R'/I(G'_b)) = 3$, and then
$\reg(R'/I(G')) = \reg(R'/(I(G'), b)) = \reg(R'/I(G'' )) = 3$.

\medskip
\textbf{Step 4.} Set $G''' = G'' - a$. We show $\reg(R'/I(G''')) = 3$.

We split into two cases.

\noindent\textit{Case 4a: $G_a$ contains an anticycle not involving $b$.}

Let $A$ be such an anticycle. Since both $A$ and $\{v_1, v_2\}$ are contained in
$G''_a$ and are disjoint, $\reg(R'/I(G''_a)) + 1 = 4$.
Hence $\reg(G''') = \reg(G'') = 3$.

\noindent\textit{Case 4b: Every anticycle in $G_a$ involves $b$.}

Let $A$ be an anticycle in $G_a$ containing $b$.
There exist two vertices $u_1, u_2$ on $A$ not adjacent to $b$.
Define $\overline{G''}$ by
\[
V(\overline{G''}) = V(G'') \cup \{v_3\},
\]
\[
E(\overline{G''}) = E(G'') \cup \bigl\{\{v_3, w\} : w \in V(G) \setminus \{u_1, u_2\}\bigr\}.
\]
Setting $R'' = R'[x_{v_3}]$ and considering the sequence
\[
0 \to \frac{R''}{(I(\overline{G''}) : v_3)}(-1) \to \frac{R''}{I(\overline{G''})}
  \to \frac{R''}{(I(\overline{G''}), v_3)} \to 0,
\]
since $\overline{G''} - v_3 = G''$ and $\overline{G''}_{v_3}$ consists of the four
vertices $u_1, u_2, v_1, v_2$, we obtain
\[
\reg(R''/I(\overline{G''})) = \reg(R'/I(G'')) = 3.
\]
Next, consider
\[
0 \to \frac{R''}{(I(\overline{G''}) : a)}(-1) \to \frac{R''}{I(\overline{G''})}
  \to \frac{R''}{(I(\overline{G''}), a)} \to 0.
\]
Since $\overline{G''}_a$ contains the induced cycle $(A - b) \cup \{v_3\}$
together with the edge $\{v_1, v_2\}$, the left-hand term has regularity $4$.
Therefore $3 = \reg(R''/I(\overline{G''} - a))$.
Removing the open neighborhood of $v_3$ then gives
$3 = \reg(R''/I(G'' - a)) = \reg(R'/I(G'''))$.

\medskip
\textbf{Step 5.} Set $H = G''' - v_2$. We show $\reg(R'/I(H)) = 3$.

Consider the two Mayer--Vietoris-type sequences
\[
0 \to \frac{R'}{T_1} \to \frac{R'}{(I(G'''), v_2)} \oplus \frac{R'}{(v_2 w : w \in N_{G'''}(v_2))}
  \to \frac{R'}{I(G''')} \to 0,
\]
\[
0 \to \frac{R'}{T_2} \to \frac{R'}{(I(H), v_1)} \oplus \frac{R'}{(v_1 w : w \in N_H(v_1))}
  \to \frac{R'}{I(H)} \to 0,
\]
where $T_1 = (I(G'''), v_2) \cap (v_2 w : w \in N_{G'''}(v_2))$
and $T_2 = (I(H), v_1) \cap (v_1 w : w \in N_H(v_1))$.

Since $(v_2 w : w \in N_{G'''}(v_2))$ and $(v_1 w : w \in N_H(v_1))$ have linear
resolutions, Lemma~\ref{bettispli} implies both sequences are Betti splittings.
Therefore
\[
\reg(G''') = \max\!\left\{\reg\!\left(\frac{R'}{I(H)}\right),\, \reg\!\left(\frac{R'}{T_1}\right) - 1\right\},
\]
\[
\reg(H) = \max\!\left\{\reg\!\left(\frac{R'}{I(H - v_1)}\right),\, \reg\!\left(\frac{R'}{T_2}\right) - 1\right\}.
\]
The minimal monomial generators of $T_1$ and $T_2$ include $v_2 e$ and $v_1 e$
respectively for $e \subset N_G(a) \cap N_G(b)$, so their regularities are at least $2$.
Moreover, $v_1 \in N_{G'''}(v_2)$ and
$(v_2 w : w \in N_G(a) \cap N_G(b)) = (T_1 : v_1)$, giving
$\reg(R'/T_1) = \reg(R'/(T_1, v_1)) = \reg(R'/T_2)$
from the exact sequence $(T_1 : v_1)(-1) \to (T_1) \to (T_1, v_1)$ and 
$(T_1, v_1) \cong (T_2) + (f ; f \in \Gamma(T_1) \text{ with }
v_1 \text{ divide f}) + (v_1) = (T_2)$. 
Therefore we obtain $3 = \reg(R'/I(G''')) = \reg(R'/I(H))$. 

\medskip
\textbf{Step 6.} We show $H$ is vertex-minimal with respect to $\reg = 3$.

Let $S = R[v_1] = R'/v_2$. Since $H - v_1 \subsetneq G$, clearly
$\reg(S/I(H - v_1)) \le 2$.
Suppose there exists $w \in V(H) \setminus \{v_1\}$ with $\reg(S/I(H - w)) = 3$.
Reversing the argument of Steps~2--5 with $H$ replaced by $H - w$ yields
$\reg(R/I(G - w)) = 3$, contradicting the vertex-minimality of $G$.
Hence $H$ is vertex-minimal.

\medskip
\textbf{Step 7.} We verify that $G$ satisfies the neighborhood condition.

Since $H^c = G^c / \{a,b\}$ (the contraction of the edge $\{a,b\}$ in $G^c$)
and $|V(H)| = |V(G)| - 1$, the induction hypothesis applies to $H$.
By hypothesis, the neighborhood $H_w$ of every vertex $w$ of $H$ is an anticycle.
We now deduce the same for $G$.

Since $H^c$ is the $1$-skeleton of a triangulation of a closed $2$-manifold
(a pseudomanifold), every edge of $H^c$ is contained in exactly two triangles.
In particular, there exists $w \in N_{H^c}(v_1)$.
Since $H_w$ is an anticycle, there exist vertices $s, t$ (other than $v_1$) with
$\{s,t\} \notin E(H)$.
At least one of $s, t$ --- say $s$ --- is not contained in $N_{G^c}(a) \cup N_{G^c}(b)$.

Since $H_s = G_s$ is an anticycle and $G_s$ contains no anticycle avoiding $t$,
a parallel argument applies with $s$ and $t$ replacing $a$ and $b$, producing a
graph $H'$.
Both $(H_v)^c$ and $(H'_a)^c$, $(H'_b)^c$ are induced cycles.
Hence $(G_a)^c$ and $(G_b)^c$ are subgraphs of induced cycles that themselves contain
induced cycles, so they must be induced cycles.

Repeating this argument for every vertex in $N_{G^c}(a) \cup N_{G^c}(b)$ and using the
neighborhoods in $H^c$ and $(H')^c$, we conclude that the neighborhood of every vertex
of $G^c$ is an induced cycle.

\medskip
\textbf{Conclusion.} By Fact~\ref{2mani}, $G^c$ is the $1$-skeleton of a triangulation
of a closed $2$-manifold.

Let $\Delta(G)$ be the independent complex of $G$. 
From the vertex-minimality of $G$, we obtain
$\widetilde{H}_2(\Delta(G)_{|W}; k) = 0$ for any $W \subset V(G)$, 
and from $\reg(R/I(G)) = 3$, 
$\widetilde{H}_2(\Delta(G); k) \neq 0$ by Lemma~\ref{hoc}. 
Hence, when $\ch(k) \neq 2$, $\Delta(G)$ is a triangulation of an orientable surface
(Fact~\ref{2homo}); when $\ch(k) = 2$, $\Delta(G)$ is a triangulation of any closed
surface.

Conversely, suppose $G^c$ contains a triangulation $\Sigma$ of an orientable closed surface as a subgraph. Let $W=V(\Sigma)$ and let $\Delta_W$ denote the restriction of the independence complex of $G$ to $W$.
Since $\Sigma \subset G^c_{|W}$, the triangles of $\Sigma$
are faces of the clique complex of $G^c_{|W}$, which equals the independence complex $(\Delta(G))_W$.
Therefore $\widetilde{H}_2((\Delta(G))_W;\, k) \supseteq \widetilde{H}_2(\Sigma;\, k) \cong k \neq 0$ by Fact~\ref{2homo}. 

Applying Hochster's formula (Lemma~\ref{hoc}) with $j =|W|$ and $i = |W| - 3$,
we obtain $\beta_{|W|-3,\, |W|}(R/I(G)) \ge 1$,
and hence $\reg(R/I(G)) \ge |W| - (|W|-3) = 3$.

\end{proof}


\bibliographystyle{plain}

\end{document}